\documentclass[11pt, a4paper]{amsart}

\usepackage{amsmath,amssymb,hyperref}
\usepackage{amsfonts}
\usepackage{enumerate}
\usepackage{amsmath, amsthm}
\usepackage{amscd}
\input xy
\xyoption{all}

      \theoremstyle{plain}
      \newtheorem{theorem}{Theorem}[section]
      \newtheorem{lemma}[theorem]{Lemma}
      \newtheorem{corollary}[theorem]{Corollary}
      \newtheorem{proposition}[theorem]{Proposition}
      \theoremstyle{definition}
      \newtheorem{definition}[theorem]{Definition}

      \theoremstyle{remark}
      \newtheorem{remark}[theorem]{Remark}

\newcommand{\su}{{\mathrm{SU}(n,1)}}
\newcommand{\SL}{{\mathrm{SL}(2n+2,\mathbb R)}}
\newcommand{\Sp}{\mathrm{Sp}(n,1)}
\newcommand{\spone}{\mathrm{Sp}(1,1)}
\newcommand{\suone}{\mathrm{SU}(1,1)}

\newcommand{\hc}{\mathbf{H}^n_{\mathbb C}}
\newcommand{\hH}{\mathbf{H}^n_{\mathbb H}}
\newcommand{\hck}{\mathbf{H}^m_{\mathbb C}}
\newcommand{\hrk}{\mathbf{H}^m_{\mathbb R}}
\def\C{\mathbb C}
\def\R{\mathbb R}

\def\H{\mathbb H}
\def\P{\mathbb P}

\def\<{\langle}
\def\>{\rangle}

\def\0{\mathbf{0}}
\def\hh#1{{{\bf H}^{#1}_{\H}}}
\def\[#1\]{\begin{eqnarray*}#1\end{eqnarray*}}

      \makeatletter
      \def\@setcopyright{}
      \def\serieslogo@{}
      \makeatother

\begin{document}

%



   \author{Joonhyung Kim}
   \address{Department of Mathematics Education, Hannam University, 70 Hannam-ro, Daedeok-gu, Daejeon 306-791, Republic of Korea}
   \email{calvary@hnu.kr}

   \author{Sungwoon Kim}
   \address{Center for Mathematical Challenges,
   KIAS, Hoegiro 85, Dongdaemun-gu,
   Seoul, 130-722, Republic of Korea}
   \email{sungwoon@kias.re.kr}






   \title[Real trace fields]{Complex and Quaternionic hyperbolic Kleinian groups with real trace fields}


\begin{abstract}
Let $\Gamma$ be a nonelementary discrete subgroup of $\su$ or $\Sp$. We show that if the trace field of $\Gamma$ is contained in $\mathbb R$, $\Gamma$ preserves a totally geodesic submanifold of constant negative sectional curvature.
Furthermore if $\Gamma$ is irreducible, $\Gamma$ is a Zariski dense irreducible discrete subgroup of $\mathrm{SO}(n,1)$ up to conjugation.
This is an analog of a theorem of Maskit for general semisimple Lie groups of rank $1$.
\end{abstract}



\footnotetext[3]{The first author was supported by the Basic Science Research Program through the National Research Foundation of Korea (NRF) funded by the Ministry of Education, Science and Technology (NRF-2014R1A1A1004122)}

\footnotetext[4]{The second author was supported by the Basic Science Research Program through the National Research Foundation of Korea (NRF) funded by the Ministry of Education, Science and Technology (NRF-2012R1A1A2040663)}



   \thanks{}
   \thanks{}

   \dedicatory{}

   \date{}


   \maketitle



\section{Introduction}

The main algebraic objects associated to a Kleinian group $\Gamma$, that is a discrete subgroup of $\mathrm{PSL}(2,\mathbb C)$, are its invariant trace field and invariant quaternion algebra. They have played an important role in studying the arithmetic aspects of Kleinian groups, especially of finite-covolume Kleinian groups. For example, it turned out that the invariant trace field of a finite-covolume Kleinian group is a number field i.e., a finite extension of $\mathbb Q$, and the matrix entries of the elements of a finite-covolume Kleinian group are in its trace field.
The trace field of $\Gamma$ is not a commensurability invariant but its invariant trace field and invariant quaternion algebra are commensurability invariants. Note that an arithmetic Kleinian group is determined up to commensurability by its invariant trace field and invariant quaternion algebra (see \cite{NR}).


McReynolds \cite{Mc} introduced the invariant trace field and the invariant algebra for subgroups of $\mathrm{PSU}(n,1)$ in a similar way as the $\mathrm{PSL}(2,\mathbb C)$ case. Moreover he proved that they are commensurability invariants as Kleinian groups. A central theme in this theory is to study the (invariant) trace field and invariant algebra associated to a subgroup of $\mathrm{PSU}(n,1)$. 
However still very little is known about these algebraic invariants associated to complex hyperbolic Kleinian groups.
In particular, Cunha-Gusevskii \cite{CG} and Genzmer \cite{Ge} studied whether a discrete subgroup of $\mathrm{SU}(2,1)$ can be realized over its trace field. 
One aim of this paper is to understand the algebraic and geometric features of discrete subgroups of $\su$ or $\Sp$ with real trace fields.


Maskit \cite[Theorem V.G.18]{Mas98} characterized nonelementary discrete subgroups of $\mathrm{SL}(2,\mathbb C)$ with real trace fields. More precisely, if the trace field of a nonelementary discrete subgroup $\Gamma$ of $\mathrm{SL}(2,\mathbb C)$ is real, then $\Gamma$ is conjugate to a subgroup of $\mathrm{SL}(2,\mathbb R)$. In other words, $\Gamma$ is realized over the real field $\mathbb R$ up to conjugation.
The same question was naturally raised as to which discrete subgroups of $\su$ have real trace fields.
In fact, an answer for the question has been given in low dimensional case.
Cunha-Gusevskii \cite{CG} and Fu-Li-Wang \cite{FLW12} independently showed that a nonelementary discrete subgroup of $\mathrm{SU}(2,1)$ with real trace field is conjugate to a subgroup of $\mathrm{SO}(2,1)$ or $\mathrm{S}(\mathrm{U}(1)\times \mathrm{U}(1,1))$. Kim-Kim \cite{KK14} also proved that a nonelementary discrete subgroup of $\mathrm{SU}(3,1)$ with real trace field is conjugate to a subgroup of $\mathrm{SO}(3,1)$ or $\mathrm{SU}(2)\times \mathrm{SU}(1,1)$. Note that it seems not easy to extend their approaches to the general case since 
all proofs in \cite{CG,FLW12,KK14} have been actually based on matrix computations.

In the $\su$ case, the trace field of a discrete subgroup being real does not imply that the discrete subgroup is realized over $\mathbb R$ up to conjugation as the $\mathrm{SL}(2,\mathbb C)$ case. Here is a counterexample.
Let $F_2$ be a free group with two generators. Let us take a discrete faithful representation $\rho_1 : F_2 \rightarrow \suone$ corresponding to a complete hyperbolic structure on a punctured torus and any representation $\rho_2 : F_2 \rightarrow \mathrm{SU}(2)$. Define a representation $\rho : F_2 \rightarrow \mathrm{SU}(3,1)$ by $\rho =\rho_1 \oplus \rho_2$.
Then it is easy to check that $\rho(F_2)$ is a nonelementary discrete subgroup and moreover the trace field of $\rho(F_2)$ is real since every element of $\suone$ and $\mathrm{SU}(2)$ has real trace. However one can easily make $\rho(F_2)$ not to be realized over $\mathbb R$ up to conjugation by choosing proper representations $\rho_1$ and $\rho_2$. In fact since the choice of $\rho_2$ is completely free, one can construct many discrete subgroups of $\su$ in this way which have real trace fields but are not realized over $\mathbb R$ up to conjugation. For this reason, in the $\su$ case, the trace field being real dose not seem to encode algebraic properties of discrete subgroups. On the other hand, in a geometric point of view, all the previous results so far give a consistent geometric feature telling that a discrete subgroup with real trace field preserves a totally geodesic submanifold of constant negative sectional curvature like Fuchsian groups.
In the general setting of $\su$, we figure it out the geometric feature of discrete subgroups with real trace fields as follows.

\begin{theorem}\label{thm:1.1}
Let $\Gamma$ be a nonelementary discrete subgroup of $\su$. If the trace field of $\Gamma$ is real, then $\Gamma$ preserves a totally geodesic submanifold of constant negative sectional curvature in $\mathbf H_{\mathbb C}^n$.
\end{theorem}

Assuming that the symmetric metric on the complex hyperbolic $n$-space $\mathbf H_{\mathbb C}^n$ is normalized so that its sectional curvature lies between $-4$ and $-1$, it is well known that a totally geodesic submanifold of constant negative sectional curvature is isometric to either a real hyperbolic space of constant sectional curvature $-1$ or a real hyperbolic $2$-plane of constant sectional curvature $-4$. Note that the first one is isometric to $\mathbf H_{\mathbb R}^k$ for some $2\leq k \leq n$ and the second one is isometric to $\mathbf H_{\mathbb C}^1$. Theorem \ref{thm:1.1} is a generalized version of the theorem of Maskit \cite[Theorem V.G.18]{Mas98} for $\su$ in geometric aspect.

We remark here that Fu-Xie \cite{FX} gave an sufficient condition for discrete subgroups of $\su$ to preserve a $2$-dimensional totally geodesic submanifold in $\mathbf H_{\mathbb C}^n$. More precisely, they proved that if $\Gamma$ is a nonelementary discrete subgroup of $\su$ and all eigenvalues are real for every loxodromic element of $\Gamma$, then $\Gamma$ preserves a $2$-dimensional totally geodesic submanifold in $\mathbf H_{\mathbb C}^n$. However the sufficient condition given by them is too sufficient in a sense that there are too many nonelementary discrete subgroups $\Gamma$ of $\su$ so that $\Gamma$
preserves a $2$-dimensional totally geodesic submanifold in $\mathbf H_{\mathbb C}^n$ but dose not satisfy that
all eigenvalues are real for every loxodromic element of $\Gamma$.

In this paper, we also study discrete subgroups of $\Sp$ with real trace fields as well as $\su$. Since the division ring $\mathbb H$ of quaternions is not commutative, the situation is quite different from the $\su$ case. For instance the usual definition of trace is not invariant under conjugation in $\Sp$. Nonetheless it turns out that the trace field of subgroups of $\Sp$ provides a useful tool in characterizing discrete subgroups of $\Sp$ preserving totally geodesic submanifolds of constant negative sectional curvature in $\mathbf H_{\mathbb H}^n$ which is not isometric to $\mathbf H_{\mathbb H}^1$.

Let $\Gamma$ be a discrete subgroup of $\Sp$. Following the definition of the trace field as usual, one can obtain 
the skew field generated by the traces of all the elements of $\Gamma$ over $\mathbb Q$. We call this skew field \emph{the trace field} of $\Gamma$.
Note that the trace field of $\Gamma$ may be not commutative and is not invariant under conjugation in $\Sp$. 
Kim \cite{Kim13} showed that a nonelementary discrete subgroup of $\mathrm{Sp}(2,1)$ with real trace field preserves a copy of $\mathbf H_{\mathbb R}^2$ or $\mathbf H_{\mathbb C}^1$ in $\mathbf H_{\mathbb H}^2$.
In accordance with his result, we expect that the trace field of $\Gamma$ being real gives a specific geoemtric property on $\Gamma$ like the $\su$ case and we finally obtain an analogous theorem of Theorem \ref{thm:1.1} for $\Sp$.

\begin{theorem}\label{thm:1.2}
Let $\Gamma$ be a nonelementary discrete subgroup of $\Sp$. If the trace field of $\Gamma$ is real, then $\Gamma$ preserves a totally geodesic submanifold of constant negative sectional curvature  in $\mathbf H_{\mathbb H}^n$ which is not isometric to $\mathbf H_{\mathbb H}^1$.
\end{theorem}

Both Theorem \ref{thm:1.1} and \ref{thm:1.2} imply that if a nonelementary discrete subgroup $\Gamma$ of $\su$ or $\Sp$ has real trace field, $\Gamma$ acts on a real hyperbolic space of dimension at least $2$ and thus it is regarded as a nonelementary discrete subgroup of $\mathrm{SO}(n,1)$. Hence we have


\begin{theorem}\label{thm:1.3}
Let $\Gamma$ be a nonelementary torsion-free discrete subgroup of $\su$ $($resp. $\Sp)$ for $n\geq 2$. Then the followings are equivalent.
\begin{itemize}
\item[(i)] There exists a discrete faithful representation $\rho : \Gamma \rightarrow \Sp$ such that the trace field of its image group is real.
\item[(ii)] There exists a discrete faithful represenation $\rho : \Gamma \rightarrow \mathrm{SO}(n,1)$ $($resp. $\rho : \Gamma \rightarrow \mathrm{O}(n,1))$
\end{itemize}
\end{theorem}

Let $\mathcal D(\Gamma)$ be the space of all discrete faithful representations of $\Gamma$ in $\Sp$.
Then Theorem \ref{thm:1.3} gives a necessary and sufficient condition for the existence of a representation in $\mathcal D(\Gamma)$ whose trace field is real.
From this point of view, one may get an answer to the question that given a discrete subgroup $\Gamma$ of $\Sp$, there exists a representation in $\mathcal D(\Gamma)$ such that $\rho(\Gamma)$ has real trace field.
According to Theorem \ref{thm:1.3}, any discrete subgroup of $\Sp$ with real trace field has cohomological dimension at most $n$. Furthermore in particular when $n=3$, it is well known that every nonelementary discrete subgroup of $\mathrm{SO}(3,1)$ is either a hyperbolic group or a relatively hyperbolic group (see \cite{Ohs}). Hence we obtain the following corollaries immediately.

\begin{corollary}
Let $\Gamma$ be a nonelementary discrete subgroup of $\su$ or $\Sp$ for $n\geq 2$. If the virtual cohomological dimension of $\Gamma$ is greater than $n$, then the trace field of any discrete subgroup of $\su$ or $\Sp$ which is isomorphic to $\Gamma$ can not be real.
\end{corollary}

\begin{corollary}
Let $\Gamma$ be a nonelementary discrete subgroup of $\mathrm{SU}(3,1)$ or $\mathrm{Sp}(3,1)$. Suppose that $\Gamma$ is neither a hyperbolic group nor relatively hyperbolic group, the trace field of any discrete subgroup of  $\mathrm{SU}(3,1)$ or $\mathrm{Sp}(3,1)$ isomorphic to $\Gamma$ can not be real.
\end{corollary}

In $\mathrm{SL}(2,\mathbb C)$ case, it is not difficult to see that the condition for a discrete group being nonelementary is equivalent to the condition for a discrete group being irreducible. Hence one can restate the Maskit's theorem as follows: If the trace field of an irreducible discrete subgroup $\Gamma$ of $\mathrm{SL}(2,\mathbb C)$ is real, 
$\Gamma$ is conjugate to a subgroup of $\mathrm{SL}(2,\mathbb R)$. In this viewpoint, we establish another generalized version of the Maskit's theorem for $\su$ and $\Sp$ in algebraic aspect.

\begin{theorem}\label{thm:1.6}
Let $\Gamma$ be an irreducible discrete subgroup of $\su$ $($resp. $\Sp )$. Then the trace field of $\Gamma$ is real if and only if $\Gamma$ is conjugate to a Zariski dense discrete subgroup of $\mathrm{SO}(n,1)$ $($resp. $ \mathrm{O}(n,1))$.
\end{theorem}





{\bf Acknowledgement.} We would like to thank Dave Witte Morris for his helpful discussions.

\section{Preliminaries}

\subsection{Complex hyperbolic spaces}

Let $\C^{n,1}$ be a complex vector space of dimension $n+1$ with a
Hermitian form of signature $(n,1)$. An element of $\C^{n,1}$ is a
column vector $z=(z_1,\ldots,z_{n+1})^t$. In what follows,
we choose the Hermitian form on $\C^{n,1}$ given by the matrix $I_{n,1}$
$$
I_{n,1}=\left[\begin{matrix} I_n & 0 \\ 0 & -1
\end{matrix}\right].
$$
Thus $\<z,w\>=w^*I_{n,1}z=z_1 \overline{w}_1+z_2
\overline{w}_2+\cdot\cdot\cdot+z_n
\overline{w}_n-z_{n+1}
\overline{w}_{n+1}$, where $w^*$ is the Hermitian transpose of $w$.

Let $\P:\C^{n,1}\setminus\{\0\} \rightarrow {\C}P^{n}$ be the
canonical projection onto a complex projective space. Consider the
following subspaces in $\C^{n,1}$;
\[
V_0 \, &=& \{ z\in {\C}^{n,1}- \{ \0 \}\ \ |\ \ \<z,z\>=0 \ \},\\
V_- &=& \{ z\in {\C}^{n,1}\ \ |\ \ \<z,z\> < 0 \ \}.
\]
The $n$-dimensional \emph{complex hyperbolic space $\hc$} is
defined as $\P(V_-)$. The \emph{boundary} $\partial\hc$ is
defined as $\P(V_0)$.

For a vector $v$ in $\C^{n,1}\setminus\{{\mathbf 0}\}$, we shall
use the notation $\hat{v}$ to denote the point $\P(v)$ in $\C
P^n$. If a point $p$ in $\C P^n$ is given, the inverse space
$\P^{-1}(p)$ is of $1$-dimensional. We shall denote a vector in
$\P^{-1}(p)$ as $\tilde{p}$ in the situation that the choice of a
vector in $\P^{-1}(p)$ makes no confusion likewise when the
definition of the Bergmann metric is given below.

The \emph{Bergmann metric} on $\hc$ is given by the distance
formula;
$$
\cosh^2\left(\frac{\rho(p,q)}{2}\right)=
\frac{\<\tilde{p},\tilde{q}\> \<\tilde{q},\tilde{p}\>}
{\<\tilde{p},\tilde{p}\> \<\tilde{q},\tilde{q}\>},
$$
for $p,q\in\hc$. Notice that any complex multiplication on
$\tilde{p}$, or on $\tilde{q}$ in the right hand side of the above
relation will make no difference on its value;
$$
\frac{\<\lambda\tilde{p},\tilde{q}\>
\<\tilde{q},\lambda\tilde{p}\>}
{\<\lambda\tilde{p},\lambda\tilde{p}\> \<\tilde{q},\tilde{q}\>}
=\frac{\lambda\<\tilde{p},\tilde{q}\>
\bar{\lambda}\<\tilde{q},\tilde{p}\>}
{\lambda\bar{\lambda}\<\tilde{p},\tilde{p}\>
\<\tilde{q},\tilde{q}\>} =\frac{\<\tilde{p},\tilde{q}\>
\<\tilde{q},\tilde{p}\>} {\<\tilde{p},\tilde{p}\>
\<\tilde{q},\tilde{q}\>}.
$$

Let $\mathrm{U}(n,1)$ be the unitary group corresponding to the Hermitian form. Then the holomorphic isometry group of $\hc$ is the projective unitary group $\mathrm{PU}(n,1)$ and the full isometry group of $\hc$ is generated by $\mathrm{PU}(n,1)$ and complex conjugation. We denote by $\su$ the subgroup of linear transformations in $\mathrm{U}(n,1)$ with determinant $1$. We notice that this group acts transitively by isometries on  $\hc$. Then the usual trichotomy which classifies isometries of real hyperbolic spaces also holds here. That is;
\begin{itemize}
\item An isometry is \emph{loxodromic} if it fixes exactly two points of $\partial\hc$.
\item An isometry is \emph{parabolic} if it fixes exactly one point of $\partial\hc$.
\item An isometry is \emph{elliptic} if it fixes at least one point of $\hc$.
\end{itemize}

In $\hc$, it is well known that there are two types of totally geodesic submanifolds $\mathbf{H}^k_{\mathbb C}$ and $\mathbf{H}^k_{\mathbb R}$. Note that a totally geodesic submanifold of constant negative sectional curvature is either of the form $\mathbf{H}^k_{\mathbb R}$ or $\mathbf H_{\mathbb C}^1$.
We say that a discrete group is \emph{elementary} if its limit set consists of at most two points, and the others are called \emph{nonelementary}. 

\begin{definition}
Let $\Gamma$ be a subgroup of $\su$. Then the \emph{trace field} of $\Gamma$ is defined as the field generated by the traces of all the elements of $\Gamma$ over the base field $\mathbb{Q}$ of rational numbers.
\end{definition}

See \cite{God99} and \cite{Mc} for more details about the trace field.

\subsection{Quaternionic hyperbolic spaces}
Let $\H^{n,1}$ be a quaternionic vector space of dimension $n+1$ with a
Hermitian form of signature $(n,1)$. An element of $\H^{n,1}$ is a
column vector $p=(p_1,\ldots,p_{n+1})^t$. As in the complex hyperbolic case, we choose
the Hermitian form on $\H^{n,1}$ given by the matrix $I_{n,1}$
$$
I_{n,1}=\left[\begin{matrix} I_n & 0 \\ 0 & -1
\end{matrix}\right].
$$
Thus $\<p,q\>=q^*I_{n,1}p=\overline{q}^tI_{n,1}p=\overline{q}_1p_1+\overline{q}_2p_2+\cdot\cdot\cdot+\overline{q}_np_n-\overline{q}_{n+1}p_{n+1}$,
where $p=(p_1,\ldots,p_{n+1})^t$, $q=(q_1,\ldots,q_{n+1})^t \in \H^{n,1}$. The group $\Sp$ is the subgroup of $\mathrm{GL}(n+1,\mathbb H)$ which, when acting on the left, preserves the Hermitian form given above.

Let $\P:\H^{n,1}\setminus\{\0\} \rightarrow {\H}P^{n}$ be the
canonical projection onto a quaternionic projective space. Consider the
following subspaces in $\H^{n,1}$;
\[
V_0 \, &=& \{ z\in {\H}^{n,1}- \{ \0 \}\ \ |\ \ \<z,z\>=0 \ \},\\
V_- &=& \{ z\in {\H}^{n,1}\ \ |\ \ \<z,z\> < 0 \ \}.
\]
The $n$-dimensional \emph{quaternionic hyperbolic space $\hH$} is
defined as $\P(V_-)$. The \emph{boundary} $\partial\hH$ is
defined as $\P(V_0)$.
There is a metric on $\hH$ called the Bergman metric and the isometry group of $\hh{n}$ with respect to this metric is
\begin{align*}
\mathrm{PSp}(n,1) &=\{[A]:A \in \mathrm{GL}(n+1,\H), \langle p,p' \rangle = \langle Ap,Ap' \rangle, p,p' \in \H^{n,1}\}\\
&= \{[A]:A \in \mathrm{GL}(n+1,\H), I_{n,1}=A^*I_{n,1}A\},
\end{align*}
where $[A]: {\H}P^{n} \rightarrow {\H}P^{n}; x\H \mapsto (Ax)\H$ for $A \in \mathrm{Sp}(n,1)$. Here we adopt the convention that the action of $\mathrm{Sp}(n,1)$ on $\hH$ is left and the action of projectivization of $\mathrm{Sp}(n,1)$ is right action. In fact $\mathrm{PSp}(n,1)$ is the quotient group by the real scalar matrices in $\Sp$. Thus it is not difficult to see that $$\mathrm{PSp}(n,1)=\Sp / \{\pm  I \}.$$

Similarly to the complex hyperbolic space, totally geodesic submanifolds of quaternionic hyperbolic space are isometric to  either $\hh{k}$, $\mathbf{H}^k_{\mathbb C}$ or $\mathbf{H}^k_{\mathbb R}$ for some $1\leq k\leq n$. Note that a totally geodesic submanifold of constant negative sectional curvature is isometric to either $\mathbf H_{\mathbb R}^k$ for some $2\leq k \leq n$, $\mathbf H_{\mathbb C}^1$ or $\mathbf H_{\mathbb H}^1$. 
The classification of isometries by their fixed points is exactly the same as in complex hyperbolic case.

\begin{definition}
Let $\Gamma$ be a subgroup of $\Sp$. Then the \emph{trace field} of $\Gamma$ is defined as the skew field generated by the traces of all the elements of $\Gamma$ over the base field $\mathbb{Q}$ of rational numbers.
\end{definition}

We say that the trace field of $\Gamma$ is \emph{real} if the trace field of $\Gamma$ is contained in $\mathbb R$.

\subsection{Zariski topology}\label{sec:zariski}

Let $\R[x_{1,1},\ldots,x_{n,n}]$ denote the set of real polynomials in the $n^2$ variables $\{x_{j,k} \ | \ 1\leq j,k \leq n \}$.
A subset $H$ of $\mathrm{SL}(n,\R)$ is called \emph{Zariski closed} if there is a subset $\mathcal S$ of $\R[x_{1,1},\ldots,x_{n,n}]$ such that $H$ is the zero locus of $\mathcal S$. In particular, when $H$ is a subgroup of $\mathrm{SL}(n,\R)$, $H$ is called a \emph{real algebraic group}.
It is a standard fact that any Zariski closed subset of $\mathrm{SL}(n,\R)$ has only finitely many components. Furthermore, a Zariski closed subgroup of $\mathrm{SL}(n,\R)$ is a $C^\infty$ submanifold of $\mathrm{SL}(n,\R)$ and so, a Lie group.

\begin{definition} The \emph{Zariski closure} of a subset $H$ of $\mathrm{SL}(n,\R)$ is the (unique) smallest Zariski closed subset of $\mathrm{SL}(n,\R)$ that contains $H$. We use $\overline H$ to denote the Zariski closure of $H$.
\end{definition}

It is well-known that if $H$ is a subgroup of $\mathrm{SL}(n,\R)$, then $\overline H$ is also a subgroup of $\mathrm{SL}(n,\R)$.

\begin{definition}
A subgroup $H$ of $\mathrm{SL}(n,\R)$ is \emph{almost Zariski closed} if $H$ is a finite-index subgroup of $\overline H$.
\end{definition}

We remark that a connected subgroup $H$ of $\mathrm{SL}(n,\R)$ is almost Zariski closed if and only if it is the identity component of a Zariski closed subgroup.

\section{Complex hyperbolic Kleinian groups}
We are concerned with subgroups of $\su$ whose trace fields are real.
Let us define a subset $\mathcal R_{su}$ of $\su$ by $$\mathcal R_{su} = \{g\in \su \ | \ tr(g) \in \mathbb R\}.$$
Then our starting observation is that $\mathcal R_{su}$ is Zariski closed in the following sense:
It is well known that complex numbers $a+ib$ can be represented by $2 \times 2$ real matrices that have the following form:
$$\left[ \begin{array}{rr} a & -b \\ b & a \end{array} \right]$$
Via this representation, one can embed $\su$ into $\SL$. Let's denote the embedding by $\phi : \su \rightarrow \SL$.

\begin{lemma}\label{lem:realtrace}
$\phi(\mathcal R_{su})$ is a Zariski closed subset of $\SL$.
\end{lemma}

\begin{proof}
First note that $\su$ is a Zariski closed subgroup of $\SL$. 
Let $g=(g_{m,l})$ be a matrix in $\su$ where $g_{m,l}=a_{m,l}+i b_{m,l}$ for $a_{m,l}, b_{m,l} \in \mathbb R$. Then $\phi(g)$ is written by 
$$\left[ \begin{array}{rrcrr} a_{1,1} & -b_{1,1} & \cdots & a_{1,n+1} & -b_{1,n+1} \\ b_{1,1} & a_{1,1} & \cdots & b_{1,n+1} & a_{1,n+1} \\
\vdots & \vdots & \ddots & \vdots & \vdots \\ a_{n+1,1} & -b_{n+1,1} & \cdots & a_{n+1,n+1} & -b_{n+1,n+1} \\ b_{n+1,1} & a_{n+1,1} & \cdots & b_{n+1,n+1} & a_{n+1,n+1} \end{array}\right].$$

Clearly $tr(g)=g_{1,1}+\cdots +g_{n+1,n+1}$ and hence, it is easy to see that $tr(g) \in \mathbb R$ if and only if $b_{1,1} +\cdots +b_{n+1,n+1}=0$. Since $b_{1,1} +\cdots +b_{n+1,n+1}$ corresponds to a real polynomial with variables in the matrix entries of $\SL$, $\mathcal R_{su}$ is a Zariski closed subset of $\SL$.
\end{proof}

We consider the Zariski topology on $\SL$ and then the pullback topology on $\su$ under the embedding $\phi : \su \rightarrow \SL$.
Let $\Gamma$ be a subgroup of $\su$ whose trace field is contained in $\mathbb R$. Then $\Gamma$ is a subset of $\mathcal R_{su}$. Since $\mathcal R_{su}$ is Zariski closed according to Lemma \ref{lem:realtrace}, 
the Zariski closure of $\Gamma$, denoted by $\overline{\Gamma}$, is contained in $\mathcal R_{su}$. From this observation, we immediately obtain the following Corollary.

\begin{corollary}\label{Zreal}
Let $\Gamma$ be a subgroup of $\su$. Then every element of $\Gamma$ has real trace if and only if every element of $\overline \Gamma$ has real trace.
\end{corollary}

It is well known that the Zariski closure of a subgroup of $\SL$ is a Zariski closed subgroup of $\SL$ and any Zariski closed subgroup is a Lie group with finitely many connected components. In particular, the identity component is a normal subgroup and connected components are the cosets of the identity component. 

Corollary \ref{Zreal} means that it is sufficient to work with Zariski closed subgroups of $\su$ to characterize subgroups of $\su$ whose trace fields are real. If $\Gamma$ is nonelementary, its Zariski closure $\overline \Gamma$ can not have small dimension. 

\begin{lemma}
Let $\Gamma$ be a nonelementary discrete subgroup of $\su$. Then $\overline \Gamma$ has dimension at least $3$.
\end{lemma} 
\begin{proof}
Suppose that $\overline \Gamma$ has dimension at most $2$. Denote by $\overline \Gamma^\circ$ the identity component of $\overline \Gamma$. Then $\overline \Gamma^\circ$ is a connected real algebraic group with dimension at most $2$. According to \cite[Corollary 11.6]{Bor}, $\overline \Gamma^\circ$ is solvable. This implies that $\overline \Gamma$ is virtually solvable since $\overline \Gamma^\circ$ is a finite index subgroup of $\overline \Gamma$. This contradicts to the assumption that $\Gamma$ is nonelementary. Thus $\overline \Gamma$ has dimension at least $3$.
\end{proof}

To characterize nonelementary subgroups with real trace fields, we reduce the problem to characterize Zariski closed subgroups of $\su$ with dimension at least $3$ for which the trace of every element is a real number.
It is much easier to deal with Zariski closed subgroups with dimension at least $3$ than arbitrary nonelementary subgroups. This is the key idea of the paper.
We now recall the structure theorem for almost Zariski closed groups. We refer the reader to \cite[Theorem 4.4.7]{Mor} for more details.

\begin{theorem}[\cite{Mor}]\label{thm:structure}
Let $H$ be a connected subgroup of $\mathrm{SL}(m,\mathbb R)$ that is almost Zariski closed. Then there exist:
\begin{itemize}
\item a semisimple subgroup $L$ of $H$,
\item a torus $T$ in $H$, and
\item a unipotent subgroup $U$ of $H$,
\end{itemize}
such that
\begin{itemize}
\item[1)] $H=(LT)\ltimes U$,
\item[2)] $L, T$, and $U$ are almost Zariski closed, and
\item[3)] $L$ and $T$ centralize each other and have finite intersection.
\end{itemize}
\end{theorem}

The structure theorem above allows us to look at the finer structure of almost Zariski closed subgroups in the case of $\su$ as follows.

\begin{lemma}\label{semisimple}
Let $H$ be a nonamenable, connected, almost Zariski closed subgroup of $\su$. Let $H=(LT)\ltimes U$ be the decomposition as in Theorem \ref{thm:structure}. Then $L$ is a connected, noncompact, semisimple Lie group with real rank $1$ and moreover $U$ is trivial.
\end{lemma}

\begin{proof}
The connectedness of $L$ follows from the connectedness of $H$. The possible value for the real rank of $L$ is either $0$ or $1$ since $\su$ has real rank $1$. If the real rank of $L$ is $0$, then $L$ is compact. In this case, since all of $L$, $T$ and $U$ are amenable, $H$ is amenable. This contradicts to the assumption that $H$ is not amenable. Thus the real rank of $L$ should be $1$. Note that this is equivalent that $L$ is noncompact.

Now we will prove that $U$ is trivial. Assume that $U$ is not trivial. Then since every nontrivial unipotent element of $\su$ is a parabolic isometry acting on $\hc$, each element of $U$ has a unique fixed point on $\partial \hc$. Moreover, since $U$ is a unipotent subgroup of $\su$, $U$ also has a unique fixed point on $\partial \hc$. Let $\xi$ be the unique fixed point of $U$.
Due to the fact that $U$ is a normal subgroup of $H$, we have $U=lUl^{-1}$  for all $l \in L$. This implies that $U$ fixes $l(\xi)$ for all $l\in L$. Noting that $\xi$ is the unique fixed point of $U$, it holds that $l(\xi)=\xi$ for all $l\in L$. In other words, $L$ is contained in the stabilizer subgroup of $\xi$ in $\su$. However, the subgroup of $\su$ stabilizing $\xi$ is amenable and thus this contradicts to the fact that $L$ is not amenable. Therefore, $U$ is trivial.
\end{proof}

\begin{remark}
Lemma \ref{semisimple} actually works for any nonamenable connected almost Zariski closed subgroup of a rank $1$ semisimple Lie group. We will use Lemma \ref{semisimple} in the $\Sp$ case later.
\end{remark}

Let $A$ and $B$ be matrices of size $m\times n$ and $r \times s$ respectively. Then recall that the direct sum of $A$ and $B$, denoted by $A \oplus B$, is a matrix of size $(m+r)\times (n+s)$ defined as 
$$A \oplus B = \left[ \begin{array}{cccccc} a_{1,1} & \cdots & a_{1,n} & 0 & \cdots & 0 \\ \vdots & \ddots & \vdots & \vdots & \ddots & \vdots \\ a_{m,1} & \cdots & a_{m,n} & 0 & \cdots & 0 \\ 0 & \cdots & 0 & b_{1,1} & \cdots & b_{1,s} \\  \vdots & \ddots & \vdots & \vdots & \ddots & \vdots \\ 0 & \cdots & 0 & b_{r,1} & \cdots & b_{r,s} \end{array} \right].$$
Let $I_m$ denote the identity matrix of size $m$. Note that $\overline \Gamma^\circ$ is a connected, almost Zariski closed subgroup and so $\overline \Gamma^\circ$ admits the decomposition as in Theorem \ref{thm:structure}.

\begin{proposition}\label{L}
Let $\Gamma$ be a nonelementary discrete subgroup of $\su$ with real trace field.
Let $\overline \Gamma^\circ = (LT) \ltimes U$ be the decomposition as in Theorem \ref{thm:structure}. Then, the followings hold.
\begin{itemize}
\item The noncompact factor of $L$ is conjugate to either $I_{n-m} \oplus\mathrm{SO}(m,1)$ for $m\geq 2$ or $I_{n-1}\oplus\suone$,
\item $T$ is conjugate to a torus in $\mathrm{SO}(n-m) \oplus I_{m+1}$ if the noncompact factor of $L$ is conjugate to $I_{n-m} \oplus\mathrm{SO}(m,1)$ and $\mathrm{SO}(n-1) \oplus I_{2}$ if the noncompact factor of $L$ is conjugate to $I_{n-1} \oplus\mathrm{SU}(1,1)$.
\item $U$ is trivial.
\end{itemize}
\end{proposition}

\begin{proof}
Since $\Gamma$ is nonelementary, $\overline \Gamma^\circ$ can not be amenable. Applying Lemma \ref{semisimple} to $\overline \Gamma^\circ$, it follows that $U$ is trivial and $L$ is a connected, noncompact, semisimple Lie group of real rank $1$.

To prove the first statement, note that every element of $\overline \Gamma^\circ$ has real trace according to Corollary \ref{Zreal}. Let $Y$ be the rank $1$ symmetric space associated with $L$. Then $Y$ is a totally geodesic submanifold of $\hc$. Due to the classification of totally geodesic submanifolds in $\hc$, $Y$ is isometric to either a totally complex geodesic $m$-submanifold $\mathbf{H}_{\mathbb C}^m$ or a totally real geodesic $m$-submanifold $\mathbf{H}_{\mathbb R}^m$ for some $1\leq m\leq n$. 

First suppose that $Y$ is isometric to $\hrk$. Then $L$ is contained in the stabilizer subgroup of $Y$ in $\su$. Hence by conjugation, we may assume that $L$ is contained in $\mathrm{SU}(n-m)\oplus \mathrm{SO}(m,1)$.
Let $L_c$ and $L_{nc}$ be the compact and noncompact factors of $L$ respectively. Then $L_c \subset \mathrm{SU}(n-m)$ and $L_{nc} \subset \mathrm{SO}(m,1)$. Note that the symmetric space associated with $L_{nc}$ is also $Y$. Hence $L_{nc}$ is a connected semisimple Lie group isogenous to $\mathrm{SO}(m,1)$. Noting that any connected semisimple Lie group is almost Zariski closed, it is easy to see that $L_{nc}$ is an almost Zariski closed subgroup of $\mathrm{SO}(m,1)$ of finite index. Since $\mathrm{SO}(m,1)$ is connected, it has no almost Zariski closed subgroups of finite index. Therefore $L_{nc}$ has to be equal to $\mathrm{SO}(m,1)$.
Recall that it is required that $L_{nc}$ is not amenable and every element of $L_{nc}$ has real trace.
If $m=1$, then $\mathrm{SU}(n-1)\oplus \mathrm{SO}(1,1)$ is amenable and thus $m \geq 2$. Since every element of $\mathrm{SO}(m,1)$ has real trace, $\mathrm{SO}(m,1)$ for $m\geq 2$ is a possible semisimple Lie group for $L_{nc}$.

Next we suppose that $Y$ is isometric to $\hck$. As the previous case, it can be easily seen that $L$ is contained in $\mathrm{SU}(n-m)\oplus \mathrm{SU}(m,1)$ and $L_{nc}=\mathrm{SU}(m,1)$ by conjugation.
Since every element of $L_{nc}$ must have real trace, the trace of every element of $\mathrm{SU}(m,1)$ has to be a real number. 
This is possible only when $m=1$. Thus $L_{nc}$ is conjugate to $\suone$.

Now the second statement only remains. Since $L$ and $T$ centralize each other, $T$ also stabilizes the symmetric space $Y$. 
Hence $T$ is contained in either $\mathrm{SU}(n-m)$ by conjugation. We set $r=n-m$. Since any torus is contained in a maximal torus, by conjugation we assume that $T$ is a torus contained in the maximal torus $T_{max}$ defined by
$$T_{max}=\{e^{i\theta_1}\oplus \cdots \oplus e^{i\theta_r} \ | \ e^{i\theta_1} \cdots e^{i\theta_r}=1, \ \forall i, \ \theta_i \in \mathbb R\}.$$ 
Let $S$ be an one-dimensional torus in $T$. Then $S$ can be written as 
$$S=\{e^{ia_1t}\oplus \cdots \oplus e^{ia_rt} \ | \ e^{ia_1t} \cdots e^{ia_rt}=1, \ t \in \mathbb R\}$$ for some $(a_1,\ldots,a_r) \in \mathbb R^r$. In order that every element of $S$ has real trace, for all $t\in \mathbb R$ it holds that
$$ \sin a_1t + \cdots + \sin a_rt =0.$$
Differentiating both sides repeatedly with respect to $t$ and subsitituting $t=0$, it is easy to see that for all integer numbers $s\geq 0$, 
\begin{eqnarray}\label{eqnsol} a_1^{2s+1}+\cdots +a_{r}^{2s+1}=0.\end{eqnarray}
It is not difficult to see that up to ordering, any solution for (\ref{eqnsol}) is of the form $$(a_1,-a_1,\ldots,a_k,-a_k,0,\ldots,0).$$
Thus $S=\{R(a_1t) \oplus \cdots \oplus R(a_k t) \oplus I_{r-2k} \ | \ t \in \mathbb R\}$ where $R(\theta ) = e^{i\theta }\oplus e^{-i\theta}$. 

By a similar way to one-dimensional torus case, it can be shown that any torus of $T_{max}$ in which every element has  real trace is of the form
$$S_1\oplus \cdots \oplus S_l$$
where $S_i$ is an one-dimensional torus of the form $R(a_1t) \oplus \cdots \oplus R(a_k t)$.

Noting that $$\left[ \begin{array}{cc} i & i \\ 1 & -1 \end{array} \right] \left[ \begin{array}{cc} e^{i\theta} & 0 \\ 0 & e^{-i\theta}  \end{array} \right] \left[ \begin{array}{cc} i & i \\ 1 & -1 \end{array} \right]^{-1} = \left[ \begin{array}{rr} \cos \theta & -\sin \theta \\ \sin \theta & \cos \theta \end{array} \right],$$
it easily follows that $T$ is conjugate to a torus in $\mathrm{SO}(r)$.
\end{proof}

\begin{remark}
In the proof of Proposition \ref{L}, we use the fact that the trace is invariant under conjugation in $\mathrm{GL}(n,\mathbb C)$. However the usual definition of trace is not invariant under conjugation in $\mathrm{GL}(n,\mathbb H)$. Hence the proof as in Proposition \ref{L} is not available in the case of $\Sp$. 
\end{remark}

\begin{theorem}\label{realtrace}
Let $\Gamma$ be a nonelementary discrete subgroup of $\su$ with real trace field. Then $\Gamma$ is conjugate to a subgroup of either $\mathrm{S}(\mathrm{U}(n-m)\oplus \mathrm{O}(m,1))$ for $m\geq 2$ or $\mathrm{SU}(n-1) \oplus \mathrm{SU}(1,1)$.
\end{theorem}

\begin{proof}
The Zariski closure $\overline \Gamma$ of $\Gamma$ is a Lie group with finitely many components. Hence its identity component $\overline \Gamma^\circ$ is a finite index normal subgroup of $\overline \Gamma$ and moreover $\overline \Gamma$ can be written as $$\overline \Gamma = \bigcup_{k=1}^r \gamma_k \overline \Gamma^\circ.$$

According to Proposition \ref{L}, $\overline \Gamma^\circ$ stabilizes a totally geodesic submanifold $Y$ of $\hc$ which is isometric to either $\hrk$ for $m\geq 2$ or $\mathbf{H}^1_{\mathbb C}$. Since $\overline \Gamma^\circ$ is a normal subgroup of $\overline \Gamma$, we have $$\overline \Gamma^\circ = \gamma_k \overline \Gamma^\circ \gamma_k^{-1}$$ for all $\gamma_k$. This means that $\overline \Gamma^\circ$ stabilizes $\gamma_k(Y)$. Since $\overline \Gamma^\circ$ can not stabilize two distinct copies of $Y$, we have $\gamma_k(Y)=Y$, that is, $\gamma_k$ also stabilizes $Y$ for all $k=1,\ldots,r$ and so does $\overline \Gamma$. 

In the case that $Y$ is isometric to $\hrk$, the stabilizer group of $Y$ in $\su$ is conjugate to $\mathrm{S}(\mathrm{U}(n-m)\oplus \mathrm{O}(m,1))$. Hence $\Gamma \subset \mathrm{S}(\mathrm{U}(n-m)\oplus \mathrm{O}(m,1))$ by conjugation.
If $Y$ is isometric to $\mathbf{H}^1_{\mathbb C}$, the stabilizer group of $Y$ in $\su$ is conjugate to $\mathrm{S}(\mathrm{U}(n-1)\oplus \mathrm{U}(1,1))$ and so $\gamma_i \in \mathrm{S}(\mathrm{U}(n-1)\oplus \mathrm{U}(1,1))$ by conjugation.
Write $\gamma_k=\gamma_k^c \oplus \gamma_k^{nc}$ for $\gamma_k^c \in \mathrm{U}(n-1)$ and $\gamma_k^{nc} \in \mathrm{U}(1,1)$. 
As shown in the proof of Proposition \ref{L}, $L_{nc}=I_{n-1}\oplus \mathrm{SU}(1,1) \subset \overline \Gamma^\circ$. Hence the trace of every element of $\gamma_k (I_{n-1}\oplus \mathrm{SU}(1,1))$ must be a real number. Let $\gamma_k^{nc} \in e^{i\theta_k} \suone$. Then we have 
$$ tr(\gamma_k^c) + t e^{i\theta_k} \in \mathbb R \text{ for all } t \in [-2,2].$$
This implies that $tr(\gamma_k^c) \in \mathbb R$ and $e^{i\theta_k} \in \mathbb R$ i.e. $e^{i\theta_k} =\pm 1$. Therefore we can conclude that $\gamma_k^{nc} \in \suone$ and $\gamma_k^c \in \mathrm{SU}(n-1)$. This completes the proof.
\end{proof}

Now we are ready to prove Theorem \ref{thm:1.1}

\begin{proof}[Proof of Theorem \ref{thm:1.1}]
According to Theorem \ref{realtrace}, $\Gamma$ preserves a totally geodesic submanifold which is isometric to $\mathbf H_{\mathbb R}^m$ for some $m\geq 2$ or $\mathbf H_{\mathbb C}^1$. These totally geodesic submanifolds have constant negative sectional curvature. Therefore Theorem \ref{thm:1.1} immediately follows.
\end{proof}

In particular when $\Gamma$ is irreducible, the possible Lie group for $\overline \Gamma^\circ$ is only $\mathrm{SO}(n,1)$. Hence Theorem \ref{thm:1.6} follows in the $\su$ case.

\begin{theorem}
Let $\Gamma$ be an irreducible discrete subgroup of $\su$. Then the trace field of $\Gamma$ is real if and only if $\Gamma$ is conjugate to a Zariski dense discrete subgroup of $\mathrm{SO}(n,1)$.
\end{theorem}

In the case of $n=2$, we obtain a stronger version of the Fu-Li-Wang's theorem in \cite{FLW12} as a Corollary.

\begin{corollary}\label{cor1}
Let $\Gamma$ be a nonelementary discrete subgroup of $\mathrm{SU}(2,1)$. Then the trace field of $\Gamma$ is real if and only if $\Gamma$ is conjugate to a subgroup of either $\mathrm{SO}(2,1)$ or $1 \oplus\mathrm{SU}(1,1)$.
\end{corollary}
\begin{proof}
It follows from Theorem \ref{realtrace} that $\Gamma$ is conjugate to a subgroup of $\mathrm{SO}(2,1)$ or $\mathrm{SU}(1) \oplus\mathrm{SU}(1,1)$. Since $\mathrm{SU}(1)=\{1\}$, the corollary immediately follows. The converse is trivial.
\end{proof}

Note that $\mathrm{SU}(2)$ is the following group:
$$\mathrm{SU}(2)= \left\{ \left[\begin{array}{rr} \alpha & - \bar \beta \\ \beta &\bar \alpha \end{array} \right] \bigg| \ \alpha, \beta \in \mathbb C, \ |\alpha|^2+|\beta|^2=1 \right\}.$$ 
Hence every element of $\mathrm{SU}(2)$ has real trace. From this fact, when $n=3$, we also get the result in \cite{KK14} as a Corollary.

\begin{corollary}\label{cor2}
Let $\Gamma$ be a nonelementary discrete subgroup of $\mathrm{SU}(3,1)$. Then the trace field of $\Gamma$ is real if and only if $\Gamma$ is conjugate to a subgroup of either $\mathrm{SO}(3,1)$ or $\mathrm{SU}(2)\oplus\mathrm{SU}(1,1)$.
\end{corollary}
\begin{proof}
Applying Theorem \ref{realtrace} to the case of $n=3$, it follows that $\Gamma$ is conjugate to either $\mathrm{SO}(3,1)$ or
$\mathrm{S}(\mathrm{U}(1)\oplus \mathrm{O}(2,1))$ or $\mathrm{SU}(2)\oplus\mathrm{SU}(1,1)$.
Since the determinant of every element of $\mathrm{O}(2,1)$ is $\pm 1$, it can be easily seen that
$$\mathrm{S}(\mathrm{U}(1)\oplus \mathrm{O}(2,1))=\mathrm{S}(\mathrm{O}(1)\oplus \mathrm{O}(2,1)) \subset \mathrm{SO}(3,1).$$ 
Thus $\Gamma$ is conjugate to a subgroup of either $\mathrm{SO}(3,1)$ or $\mathrm{SU}(2)\oplus\mathrm{SU}(1,1)$.
As observed above, since the trace of every element of $\mathrm{SU}(2)$ is a real number, the converse clearly holds.
\end{proof}

\section{Quaternionic hyperbolic Kleinian groups}

As seen in the previous section, the trace field is a useful tool in recognizing subgroups of $\su$ which stabilize a totally geodesic submanifold of constant negative sectional curvature.
The question is naturally raised whether it works in the setting of $\Sp$ or not.
The main difficulty in extending the argument in the $\su$ case to $\Sp$ is that the trace is not invariant under conjugation in $\Sp$. This is due to the noncommutativity of the division ring $\mathbb H$ of quaternions. In the $\su$ case, if the trace field of a subgroup $\Gamma$ of $\su$ is not real, then the trace field of any subgroup conjugate to $\Gamma$ is not real either since the trace field is invariant under conjugation. However, this does not work in $\Sp$. Even if the set of traces of a subgroup $\Gamma$ of $\Sp$ is not real, it is possible that the trace field of some subgroup conjugate to $\Gamma$ can be real. This is the main difference between $\su$ and $\Sp$. We will give such example in Section \ref{sec:sp11}
Nonetheless, Kim \cite{Kim13} gives a positive answer for $\mathrm{Sp}(2,1)$ as follows.
\begin{theorem}[Kim]
Let $\Gamma < \mathrm{Sp}(2,1)$ be a nonelementary quaternionic hyperbolic Kleinian group containing a loxodromic element fixing $0$ and $\infty$. Assume that the sum of diagonal entries of each element of $\Gamma$ is real. Then $\Gamma$ stabilizes a copy of either $\mathbf{H}_{\mathbb R}^2$ or $\mathbf{H}_{\mathbb C}^1$.
\end{theorem}

In accordance with his result, one can expect that the usual definition of trace in $\Sp$ is also useful in recognizing subgroups of $\Sp$ which preserve some specific totally geodesic submanifold.

Throughtout this section, $\Gamma$ denote a nonelementary discrete subgroup of $\Sp$ with real trace field and we will stick to the notation used in Lemma \ref{qualemma1} and denote by $L_{nc}$ the noncompact factor of $L$.

Similarly to the $\su$ case, we define a subset $\mathcal R_{sp}$ of $\Sp$ by  $$\mathcal R_{sp} = \{g\in \Sp \ | \ tr(g) \in \mathbb R\}.$$ It is a standard fact that $\Sp$ can be embedded in $\mathrm{SL}(4n, \mathbb R)$ by identifying $\mathbb H^n$ with $\mathbb R^{4n}$. More precisely, a quaternion $a+bi+cj+dk$ can be written as a $4\times 4$ real matrix, $$\left[ \begin{array}{rrrr} a & b & c & d \\ -b & a & -d & c \\ -c & d & a & -b \\ -d & -c & b & a \end{array} \right].$$
By a similar proof of Lemma \ref{lem:realtrace}, it follows that $\mathcal R_{sp}$ is Zariski closed in $\mathrm{SL}(4n,\mathbb R)$. Furthermore it is not difficult to see that Corollary \ref{Zreal} and Lemma \ref{semisimple} also work in the setting of $\Sp$ by following their proofs in the $\su$ case. Hence we have

\begin{lemma}\label{qualemma1}
Let $\Gamma$ be a nonelementary discrete subgroup of $\Sp$ with real trace field. Then every element of the Zariski closure $\overline \Gamma$ of $\Gamma$ also has real trace. Furthermore there exist a connected real rank $1$ semisimple subgroup $L$ of $\overline \Gamma^\circ$ and a torus $T$ in $\overline \Gamma^\circ$ such that $L$ and $T$ centralize each other and have finite intersection, and $\overline \Gamma^\circ=LT$.
\end{lemma}

If the trace is invariant under conjugation in $\Sp$, then one can exclude $\mathrm{SU}(m,1)$ for $2\leq m\leq n$ and $\mathrm{Sp}(k,1)$ for $1\leq k\leq n$ in the list of possible Lie groups for $L_{nc}$ as done in the $\su$ case. Unfortunately one can not do. This makes it difficult to find  all possible Lie groups for $L_{nc}$.
First, we will start with the $\mathrm{Sp}(1,1)$, $\mathrm{Sp}(2,1)$ cases and then we will deal with the general case. 


\subsection{$\mathrm{Sp}(1,1)$-case}\label{sec:sp11}

Recall that $$\mathrm{Sp}(1,1) =\{ g \in M_2(\mathbb H) \ | \ g^*I_{1,1}g=I_{1,1} \}$$ where $g^*$ is the conjugate transposed matrix of $g$.
A straightforward computation shows that
$$\mathrm{Sp}(1,1)=\Bigg\{ \left[ \begin{array}{rr} a & b \\ c & d \end{array} \right] \in M_2(\mathbb H) \ \Bigg| \ |a|^2-|c|^2=|d|^2-|b|^2=1, \ \bar a b= \bar c d \Bigg\}.$$
For a matrix $A=\left[ \begin{array}{rr} a & b \\ c & d \end{array} \right] \in \mathrm{Sp}(1,1)$, its inverse $A^{-1}$ is written as $$A^{-1}=\left[ \begin{array}{rr} \bar a & -\bar c \\ -\bar b & \bar d \end{array} \right].$$
In addition, it can be easily seen that $|a|=|d|$ and $|b|=|c|$.

Let $\Gamma$ be a nonelementary discrete subgroup of $\mathrm{Sp}(1,1)$ in which $tr(\gamma)\in \mathbb R$ for all $\gamma \in \Gamma$. According to Lemma \ref{qualemma1}, $\overline \Gamma^\circ=LT$. 
Let $Y$ be the symmetric space associated with $L$. Then $Y$ is a totally geodesic submanifold of $\mathbf H_{\mathbb H}^1$ and hence $Y$ is isometric to either $\mathbf H_{\mathbb C}^1$ or $\mathbf H_{\mathbb H}^1$. Noting that $T$ centralizes $L$ and  $\overline \Gamma^\circ$ is a normal subgroup of $\overline \Gamma$, it can be easily shown that $\overline \Gamma$ stabilizes $Y$.
If $Y= \mathbf H_{\mathbb H}^1$, then $\overline \Gamma=\mathrm{Sp}(1,1)$.
However the set of traces of $\mathrm{Sp}(1,1)$ is not contained in $\mathbb R$. For example, take 
$$g=\left[ \begin{array}{rr} i & 0 \\ 0 & j \end{array} \right] \in \mathrm{Sp}(1,1).$$ Then $tr(g)=i+j$ is not real. For this reason, $\overline \Gamma$ can never be $\mathrm{Sp}(1,1)$ and thus $Y$ can not be $\mathbf H_{\mathbb H}^1$. 

Now we suppose that $Y$ is isometric to $\mathbf H_{\mathbb C}^1$. By a similar reason as the proof of Proposition \ref{L}, the noncompact simple factor $L_{nc}$ of $L$ is conjugate to $\suone$. Let $L_{nc}=g\suone g^{-1}$ for some $g\in \spone$. Since every element of $\overline \Gamma$ has real trace, every element of $g \suone g^{-1}$ also has real trace. Note that although the set of traces of $\suone$ is contained in $\mathbb R$, the set of traces of $g \suone g^{-1}$ may be not contained in $\mathbb R$ for some $g\in \spone$. Here is an example. Let $$g=\left[ \begin{array}{rr} \frac{1+i}{\sqrt 2} & 0 \\ 0 & \frac{j+k}{\sqrt 2} \end{array} \right] \in \mathrm{Sp}(1,1).$$ Any element of $\suone$ can be written as $$\left[ \begin{array}{rr} z & w \\ \bar w & \bar z \end{array} \right]$$ where $z$ and $w$ are complex numbers with $|z|^2 - |w|^2 =1$. Then a straightfoward computation shows that

\begin{eqnarray*}
\lefteqn{tr \left( \left[ \begin{array}{cc} \frac{1+i}{\sqrt 2} & 0 \\ 0 & \frac{j+k}{\sqrt 2} \end{array} \right] \left[ \begin{array}{cc} z & w \\ \bar w & \bar z \end{array} \right] \left[ \begin{array}{cc} \frac{1+i}{\sqrt 2} & 0 \\ 0 & \frac{j+k}{\sqrt 2} \end{array} \right]^{-1} \right)} \\
& & =\frac{1}{2} \left\{ (1+i)z(1-i) -(j+k)\bar z (j+k) \right\} = 2z.
\end{eqnarray*}
Since $z$ can be any complex number with $|z| \geq 1$, the set of traces of $g \suone g^{-1}$ is equal to $\{z\in \mathbb C \ | \ |z|\geq 2 \}$ which is not contained in $\mathbb R$. 

Consider the set $\mathcal R$ of $\mathrm{Sp}(1,1)$ defined by \begin{eqnarray}\label{suone}\mathcal R=\{ g\in \spone \ | \ tr(g A g^{-1}) \in \mathbb R \text{ for all }A\in \suone \}. \end{eqnarray}
One can notice that $\suone \subset \mathcal R$. We show that the converse is also true up to the left action of $\mathrm{Sp}(1)$ on $\suone$ as follows.

\begin{proposition}
Every element of $g \suone g^{-1}$ has real trace if and only if $g \in \mathrm{Sp}(1)\cdot \suone$.
\end{proposition}
\begin{proof}
It is sufficient to show that $\mathcal R=\mathrm{Sp}(1) \cdot \suone$.
First of all, we observe that if $g \in \mathcal R$, then $hg \in \mathcal R$ for any unit quaternion number $h$ as follows: If $g \in \mathcal R$, then $tr(gug^{-1}) \in \mathbb \mathcal R$ for all $u \in \suone$. It is easy to check that for any unit quaternion number $h$,
\begin{eqnarray*}
tr((hg)u(hg)^{-1})&=&tr((hg)u(g^{-1} \bar h))\\
&=&h \cdot tr(gug^{-1}) \cdot \bar h=|h|^2 tr(gug^{-1}) = tr(gug^{-1}) \in \mathbb R\end{eqnarray*}
for all $u\in \suone$.
Hence $hg \in \mathcal R$. This implies that if $g\in \mathcal R$, then every element of the $\mathrm{Sp}(1)$-orbit of $g$ under the left action of $\mathrm{Sp}(1)$ on $\spone$ is an element of $\mathcal R$. Since the trace is invariant in $\suone$ under conjugation by any element of $\suone$, it is clear that $\suone \subset \mathcal R$. It follows from the above observation that $\mathrm{Sp}(1) \cdot \suone \subset \mathcal R$.
From now on, we will show its converse $\mathcal R \subset \mathrm{Sp}(1) \cdot \suone$.

Put $$g=\left[ \begin{array}{rr} a&b \\c&d\end{array}\right] \in \spone, \ |a|^2-|c|^2=|d|^2-|b|^2=1, \ \bar a b= \bar c d.$$
Recall that $|a||b|=|c||d|$. We may assume that $a$ is a nonnegative real number by multiplying some unit quaternion number to $g$.

\smallskip
{\bf Case $1$} : $|a||b|=|c||d|=0$

\smallskip
Due to $|a|^2-|c|^2=|d|^2-|b|^2=1$, $|a|$ and $|d|$ can never be $0$. Hence we have $b=c=0$ and $|a|=|d|=1$. Moreover since $a$ is a nonnegative real number, $a=1$. By the assumption that every element of $g \suone g^{-1}$ has real trace, 
$$tr \left( \left[ \begin{array}{cc} a&0\\0&d \end{array} \right] \left[ \begin{array}{cc} i&0\\0&-i \end{array} \right] \left[ \begin{array}{cc} \bar a&0\\0& \bar d \end{array} \right] \right)= ai\bar a - d i \bar d=  i-di\bar d \in \mathbb R.$$ 
Noting that $\overline {i-di \bar d}= -(i-di\bar d)$, it is easy to see that $i-di \bar d=0$. Writing $d=d_1+d_2i+d_3j+d_4k$,
\begin{eqnarray*}
di\bar d &=& (d_1+d_2i+d_3j+d_4k)i(d_1-d_2i-d_3j-d_4k) \\
&=& (d_1^2+d_2^2-d_3^2-d_4^2)i +2(d_1d_4+d_2d_3)j+2(-d_1d_3+d_2d_4)k =i.
\end{eqnarray*}
Now we have the following equations.
\begin{equation}\label{d}
     \begin{aligned}
& d_1^2+d_2^2-d_3^2-d_4^2=1, \\
& d_1d_4+d_2d_3=0, \\
& d_1d_3-d_2d_4=0.
     \end{aligned}
\end{equation}
Then we have $d_1d_2(d_3^2+d_4^2)=0$ from the last two equations of (\ref{d}). If $d_1$=0, then $d_2$ can never be $0$ due to $d_1^2+d_2^2-d_3^2-d_4^2=1$. Hence $d_3=d_4=0$. In a similar way, if $d_2=0$, then $d_3=d_4=0$. If $d_3^2+d_4^2=0$, then clearly, $d_3=d_4=0$. Thus in either case, $d_3=d_4=0$. This implies $d\in \mathbb C$ with $|d|=1$. Let $d=e^{2i\theta}$. Then
$$g \in \mathrm{Sp}(1) \cdot \left[ \begin{array}{cc} 1&0 \\0& e^{2i\theta} \end{array} \right] = \mathrm{Sp}(1) \cdot \left[ \begin{array}{cc} e^{-i\theta} &0 \\0& e^{i\theta} \end{array} \right] \subset \mathrm{Sp}(1)\cdot \suone.$$

\smallskip

{\bf Case $2$} : $|a||b|=|c||d|\neq 0$

\smallskip
Recall again that any element of $\suone$ can be written as $$\left[ \begin{array}{cc} z & w \\ \bar w &\bar z \end{array}\right], \ z, w \in \mathbb C, \ |z|^2-|w|^2=1.$$
Let $z =z_1 + iz_2$ and $w=w_1 + iw_2 $. Then
\begin{eqnarray*}
\lefteqn{tr \left( \left[ \begin{array}{cc} a&b\\c&d \end{array} \right] \left[ \begin{array}{cc}z & w \\ \bar w &\bar z \end{array} \right] \left[ \begin{array}{rr} \bar a& -\bar c \\ -\bar b& \bar d \end{array} \right] \right)} \\
&=& z_1(|a|^2-|b|^2-|c|^2+|d|^2) + z_2 (ai\bar a + bi \bar b- ci\bar c -di \bar d) \\
& & + w_1 (-a \bar b + b \bar a + c\bar d - d\bar c) +w_2 ( - ai \bar b -bi \bar a +ci \bar d+ di \bar c) \in \mathbb R.
\end{eqnarray*}
for all real numbers $z_1$, $z_2$, $w_1$ and $w_2$ with $z_1^2+z_2^2-w_1^2-w_2^2=1$. Hence we have 
\begin{equation}\label{case2}
     \begin{aligned}
& ai\bar a + bi \bar b- ci\bar c -di \bar d \in \mathbb R, \\
& a \bar b - b \bar a - c\bar d + d\bar c \in \mathbb R, \\
& ai \bar b +bi \bar a -ci \bar d- di \bar c \in \mathbb R.
     \end{aligned}
\end{equation}
All three numbers in (\ref{case2}) satisfy $\bar q= -q$ and so, they are actually zero. 
Since $a$ is real and $\bar a b= \bar c d$, $b= \frac{\bar c d}{a}$. Then
$$
0= a \bar b - b \bar a - c\bar d + d\bar c = \bar d c - \bar c d - c \bar d + d \bar c = (\bar d c - c \bar d) - \overline{(\bar d c - c \bar d)}.$$
This implies that $\bar d c - c \bar d$ is real. Furthermore, it can be shown by a direct computation that the real part of $\bar d c - c \bar d$ is zero for all quaternion numbers $c$, $d$ and hence $\bar d c - c \bar d=0$. 

Putting $b= \frac{\bar c d}{a}$ in $ai \bar b +bi \bar a -ci \bar d- di \bar c =0$,
$$0=ai \bar b +bi \bar a -ci \bar d- di \bar c =i \bar d c +\bar c d i - ci\bar d -di \bar c = (i\bar d c - c i \bar d) -\overline{(i\bar d c - c i \bar d)}.$$
This implies $(i\bar d c - c i \bar d)\in \mathbb R$. It is easy to check that the real part of $(i\bar d c - c i \bar d)$ is zero and thus, $i\bar d c - c i \bar d=0$. Since $\bar d c=c \bar d$, $$ i\bar d c - c i \bar d=(ic-ci)\bar d=0.$$
Finally we have $ci=ic$ since $d$ can not be $0$. It is not difficult to see that if $ci=ic$, then $c$ is a complex number.
If $c$ is a real number, then $b=\frac{c}{a}d$ and so we have
$$0=ai\bar a + bi \bar b- ci\bar c -di \bar d =(a^2-c^2)i +\left(\frac{c}{a}\right)^2 di\bar d - di \bar d=i-\frac{di\bar d}{a^2}.$$
This means that $$\frac{d}{a} i \bar{\left(\frac{d}{a}\right)}=i.$$
By the same argument in Case $1$, it follows that $\frac{d}{a}$ is a complex number and so is $d$. From $b=\frac{c}{a}d$, it follows that $b$ is a complex number. Therefore, all numbers $a$, $b$, $c$ and $d$ are complex numbers.

If $c$ is a complex number which is not real, then $\bar d c=c \bar d$ implies that $d$ is a complex number. Since $c$ and $d$ are complex numbers, $b= \frac{\bar c d}{a}$ is also a complex number. Hence all $a$, $b$, $c$ and $d$ are complex numbers.

In either Case $1$ or $2$, we prove that $g$ is a complex matrix in $\spone$ and thus, $g \in \mathrm{U}(1,1)$. It is easy to see that 
$$ g \in \mathrm{Sp}(1) \cdot \mathrm{U}(1,1) = \mathrm{Sp}(1) \cdot \mathrm{SU}(1,1). $$ 
Therefore $\mathcal R \subset \mathrm{Sp}(1) \cdot \mathrm{SU}(1,1)$. 
Finally we have $\mathcal R= \mathrm{Sp}(1) \cdot \suone $.
\end{proof}

\begin{theorem}
Let $\Gamma$ be a nonelementary discrete subgroup of $\spone$. Then the trace field of $\Gamma$ is real if and only if $\Gamma$ is conjugate to a subgroup of $\suone$ by an element of $\mathrm{Sp}(1)$.
\end{theorem}

\begin{remark}
Regarding $\Gamma$ as a nonelementary discrete subgroup of $\mathrm{PSp}(1,1)$ acting on the quaternionic hyperbolic space $\mathbf{H}_{\mathbb H}^1$, $\Gamma$ stabilizes $\mathbf{H}_{\mathbb C}^1$ up to the left action of $\mathrm{Sp}(1)$ on $\mathbf{H}_{\mathbb H}^1$. Note that the right action of $\mathrm{Sp}(1)$ on $\mathbf{H}_{\mathbb H}^1$ is trivial but the left action of $\mathrm{Sp}(1)$ on $\mathbf{H}_{\mathbb H}^1$ is not.
\end{remark}

\subsection{$\mathrm{Sp}(2,1)$-case}

Similarly to the $\su$ case, it can be seen that $L_{nc}$ is conjugate to either $\mathrm{SO}(2,1)$, $1\oplus \mathrm{SU}(1,1)$, $\mathrm{SU}(2,1)$, $1\oplus\mathrm{Sp}(1,1)$ or $\mathrm{Sp}(2,1)$ and every element of $L_{nc}$ has real trace. Then one can notice easily that $L_{nc}$ can not be conjugate to 
$\mathrm{Sp}(2,1)$. This follows from the fact that $L_{nc}=\mathrm{Sp}(2,1)$ if $L_{nc}$ is conjugate to $\mathrm{Sp}(2,1)$ and the set of traces of $\mathrm{Sp}(2,1)$ is not contained in $\mathbb R$. 

In the case of $\mathrm{SU}(2,1)$ and $1 \oplus \mathrm{Sp}(1,1)$, one can not easily deduce that $L_{nc}$ can not be $\mathrm{SU}(2,1)$ and $1 \oplus \mathrm{Sp}(1,1)$ even if the set of traces for them is not contained in $\mathbb R$. This is because the trace is not invariant under conjugation in $\mathrm{Sp}(2,1)$. One need to check whether the set of traces is not contained in $\mathbb R$ or not, for all groups conjugate to $\mathrm{SU}(2,1)$ or $1 \oplus \mathrm{Sp}(1,1)$ in $\mathrm{Sp}(2,1)$ to exclude them.

\begin{lemma}\label{sp11}
The set of traces of $g (1\oplus \mathrm{Sp}(1,1)) g^{-1}$ is not contained in $\mathbb R$ for any $g\in \mathrm{Sp}(2,1)$.
\end{lemma}
\begin{proof}
Assume that the set of traces of $g (1\oplus \mathrm{Sp}(1,1)) g^{-1}$ is contained in $\mathbb R$ for some $g\in \mathrm{Sp}(2,1)$. Put 
$$g=\left[ \begin{array}{rrr} a_{1,1}&a_{1,2}&a_{1,3} \\ a_{2,1}&a_{2,2}&a_{2,3} \\ a_{3,1}&a_{3,2}&a_{3,3} \end{array}\right].$$
For any $q \in \mathrm{Sp}(1)$, $\left[ \begin{array}{rr} q&0\\0&1\end{array}\right]$ is an element of $\spone$. Hence we have
$$ tr \left( \left[ \begin{array}{rrr} a_{1,1}&a_{1,2}&a_{1,3} \\ a_{2,1}&a_{2,2}&a_{2,3} \\ a_{3,1}&a_{3,2}&a_{3,3} \end{array}\right] \left[ \begin{array}{rrr} 1&0&0 \\ 0&q&0 \\ 0&0&1 \end{array}\right] \left[ \begin{array}{rrr} \bar a_{1,1}& \bar a_{2,1}& -\bar a_{3,1} \\ \bar a_{1,2}& \bar a_{2,2}& -\bar a_{3,2} \\ -\bar a_{1,3}& -\bar a_{2,3}& \bar a_{3,3} \end{array}\right] \right) \in \mathbb R$$
for all unit quaternion numbers $q$. By a direct computation, it can be shown that the above condition is equivalent to 
$$ a_{1,2}q \bar a_{1,2} + a_{2,2} q \bar a_{2,2} - a_{3,2} q \bar a_{3,2} \in \mathbb R$$ for all unit quaternion numbers $q$. From the fact $g \in \mathrm{Sp}(2,1)$, it holds that $$|a_{1,2}|^2 + |a_{2,2}|^2 - |a_{3,2}|^2=1.$$
\smallskip
{\bf Claim }: There does not exist a triple of quaternion numbers $(x,y,z) \in \mathbb H^3$ with $|x|^2+|y|^2-|z|^2=1$ satisfying that $x q \bar x + y q \bar y - z q \bar z \in \mathbb R$ for all unit quaternion numbers $q$.

\begin{proof}[Proof of the Claim] 
Assume that there exists a $(x,y,z) \in \mathbb H^3$ such that $x q \bar x + y q \bar y - z q \bar z \in \mathbb R$ for all unit quaternion numbers $q$. Then it is equivalent to 
\begin{equation*}
     \begin{aligned}
& xi\bar x + yi\bar y -z i \bar z \in \mathbb R, \\
& xj\bar x + yj\bar y -z j \bar z \in \mathbb R, \\
& xk\bar x + yk\bar y -z k \bar z \in \mathbb R.
     \end{aligned}
\end{equation*}

Put $x=x_1 +x_2i+x_3j+x_4k$, $y=y_1+y_2i+y_3j+y_4k$ and $z=z_1+z_2i+z_3j+z_4k$. It follows from $xi\bar x + yi\bar y -z i \bar z \in \mathbb R$ that 
\begin{equation*}
     \begin{aligned}
& (x_1^2+x_2^2-x_3^2-x_4^2)+(y_1^2+y_2^2-y_3^2-y_4^2)-(z_1^2+z_2^2-z_3^2-z_4^2)=0, \\
& (x_1x_4+x_2x_3)+(y_1y_4+y_2y_3)-(z_1z_4+z_2z_3)=0, \\
& (x_2x_4-x_1x_3)+(y_2y_4-y_1y_3)-(z_2z_4-z_1z_3)=0.
     \end{aligned}
\end{equation*}
Similarly, $xj\bar x + yj\bar y -z j \bar z \in \mathbb R$ and $xk\bar x + yk\bar y -z k \bar z \in \mathbb R$ imply that 
\begin{equation*}
     \begin{aligned}
& (x_1^2-x_2^2+x_3^2-x_4^2)+(y_1^2-y_2^2+y_3^2-y_4^2)-(z_1^2-z_2^2+z_3^2-z_4^2)=0, \\
& (x_1x_4-x_2x_3)+(y_1y_4-y_2y_3)-(z_1z_4-z_2z_3)=0, \\
& (x_1x_2+x_3x_4)+(y_1y_2+y_3y_4)-(z_1z_2+z_3z_4)=0.
     \end{aligned}
\end{equation*}
and,
\begin{equation*}
     \begin{aligned}
& (x_1^2-x_2^2-x_3^2+x_4^2)+(y_1^2-y_2^2-y_3^2+y_4^2)-(z_1^2-z_2^2-z_3^2+z_4^2)=0, \\
& (x_1x_3+x_2x_4)+(y_1y_3+y_2y_4)-(z_1z_3+z_2z_4)=0, \\
& (x_1x_2-x_3x_4)+(y_1y_2-y_3y_4)-(z_1z_2-z_3z_4)=0,
     \end{aligned}
\end{equation*}
respectively. In addition, $|x|^2+|y|^2-|z|^2=1$ gives 
$$(x_1^2+x_2^2+x_3^2+x_4^2)+(y_1^2+y_2^2+y_3^2+y_4^2)-(z_1^2+z_2^2+z_3^2+z_4^2)=1.$$
Summarizing all equations above, one get the following equations.
\begin{equation}\label{alleqns}
     \begin{aligned}
& x_1^2+y_1^2-z_1^2=x_2^2+y_2^2-z_2^2=x_3^2+y_3^2-z_3^2=x_4^2+y_4^2-z_4^2=\frac{1}{4}, \\
& x_1x_2+y_1y_2-z_1z_2=0, \\
& x_1x_3+y_1y_3-z_1z_3=0, \\
& x_1x_4+y_1y_4-z_1z_4=0, \\
& x_2x_3+y_2y_3-z_2z_3=0, \\
& x_2x_4+y_2y_4-z_2z_4=0, \\
& x_3x_4+y_3y_4-z_3z_4=0. 
     \end{aligned}
\end{equation}
Let $v_i =(x_i,y_i,z_i)$ for each $i=1,\ldots,4$. Let $\mathbb R^{2,1}$ be a $3$-dimensional Minkowski space with a bilinear form $\langle \ ,\ \rangle_{2,1}$ defined by
$$\langle (x,y,z), (x',y',z') \rangle_{2,1} = xx'+yy'-zz'.$$
Then all equations in (\ref{alleqns}) mean that $\{v_1,v_2,v_3,v_4\}$ is the set of nontrivial positive vectors which are pairwise orthogonal with respect to the bilinear form $\langle \ ,\ \rangle_{2,1}$. 
However this is impossible due to the signature of $\langle \ ,\ \rangle_{2,1}$ and the dimension of $\mathbb R^{2,1}$. Therefore there does not exist any solution for (\ref{alleqns}).
\end{proof}
The claim implies that for any $g\in \mathrm{Sp}(2,1)$, the set of traces of $g(1\oplus \spone)g^{-1}$ can be never contained in $\mathbb R$. This completes the proof.
\end{proof}

Next we move to the case that $L_{nc}$ is conjugate to $\mathrm{SU}(2,1)$.

\begin{lemma}
The set of traces of $g \mathrm{SU}(2,1) g^{-1}$ is not contained in $\mathbb R$ for any $g\in \mathrm{Sp}(2,1)$.
\end{lemma}
\begin{proof}
Assume that the set of traces of $g \mathrm{SU}(2,1) g^{-1}$ is contained in $\mathbb R$ for some $g\in \mathrm{Sp}(2,1)$. 
Let us stick to the notation for the matrix form of $g$ used in Lemma \ref{sp11}. 
For any unit complex numbers $z$, $w \in \mathbb C$, note that 
$$\left[ \begin{array}{ccc} z & 0&0\\0&w&0\\0&0&\bar z\bar w \end{array} \right] \in \mathrm{SU}(2,1).$$ 
By the assumption, we have
\begin{eqnarray*} \lefteqn{tr \left( \left[ \begin{array}{rrr} a_{1,1}&a_{1,2}&a_{1,3} \\ a_{2,1}&a_{2,2}&a_{2,3} \\ a_{3,1}&a_{3,2}&a_{3,3} \end{array}\right] \left[ \begin{array}{rrr} z & 0&0\\0&w&0\\0&0&\bar z\bar w \end{array}\right] \left[ \begin{array}{rrr} \bar a_{1,1}& \bar a_{2,1}& -\bar a_{3,1} \\ \bar a_{1,2}& \bar a_{2,2}& -\bar a_{3,2} \\ -\bar a_{1,3}& -\bar a_{2,3}& \bar a_{3,3} \end{array}\right] \right) } \\
&=& (a_{1,1} z \bar a_{1,1} + a_{2,1} z \bar a_{2,1}-a_{3,1} z \bar a_{3,1})+  (a_{1,2} w \bar a_{1,2} + a_{2,2} w \bar a_{2,2}-a_{3,2} w \bar a_{3,2}) \\ & & -(a_{1,3} \bar z\bar w \bar a_{1,3} + a_{2,3} \bar z\bar w \bar a_{2,3}-a_{3,3} \bar z\bar w \bar a_{3,3})
\end{eqnarray*}
is a real number for all $z$, $w \in \mathbb C$ with $|z|=|w|=1$.

Taking $(z,w)=(1,i)$, $(i,1)$ and $(i,i)$, we obtain sequentially the following conditions:
\begin{equation*}
     \begin{aligned}
& (a_{1,2} i \bar a_{1,2} + a_{2,2} i \bar a_{2,2}-a_{3,2} i \bar a_{3,2}) + (a_{1,3} i \bar a_{1,3} + a_{2,3} i \bar a_{2,3}-a_{3,3} i\bar a_{3,3}) \in \mathbb R, \\
& (a_{1,1} i \bar a_{1,1} + a_{2,1} i \bar a_{2,1}-a_{3,1} i \bar a_{3,1}) + (a_{1,3} i \bar a_{1,3} + a_{2,3} i \bar a_{2,3}-a_{3,3} i\bar a_{3,3}) \in \mathbb R, \\
& (a_{1,1} i \bar a_{1,1} + a_{2,1} i \bar a_{2,1}-a_{3,1} i \bar a_{3,1}) + (a_{1,2} i \bar a_{1,2} + a_{2,2} i \bar a_{2,2}-a_{3,2} i\bar a_{3,2}) \in \mathbb R.
     \end{aligned}
\end{equation*}
This implies that
\begin{equation}\label{real}
     \begin{aligned}
& (a_{1,1} i \bar a_{1,1} + a_{2,1} i \bar a_{2,1}-a_{3,1} i \bar a_{3,1}) \in \mathbb R, \\
& (a_{1,2} i \bar a_{1,2} + a_{2,2} i \bar a_{2,2}-a_{3,2} i \bar a_{3,2}) \in \mathbb R, \\
& (a_{1,3} i \bar a_{1,3} + a_{2,3} i \bar a_{2,3}-a_{3,3} i \bar a_{3,3}) \in \mathbb R.
     \end{aligned}
\end{equation}

It is not difficult to check that the real part of $a_{1,1} i \bar a_{1,1} + a_{2,1} i \bar a_{2,1}-a_{3,1} i \bar a_{3,1}$ is $0$ and hence, $a_{1,1} i \bar a_{1,1} + a_{2,1} i \bar a_{2,1}-a_{3,1} i \bar a_{3,1}=0$. Similarly, the other terms in (\ref{real}) also equal to $0$.
Since $|a_{1,1}|^2+|a_{2,1}|^2 - |a_{3,1}|^2=|a_{1,2}|^2+|a_{2,2}|^2 - |a_{3,2}|^2=1$ and $|a_{1,3}|^2+|a_{2,3}|^2 - |a_{3,3}|^2=-1$, one can notice that $(a_{1,3}, a_{2,3}, a_{3,3}) \in \mathbb H^3$ is a solution of the following equation system for quaternion numbers.
\begin{equation}\label{eqnsys}
\left\{\begin{aligned}
& xi \bar x +y i \bar y - z i \bar z=0 \\
& |x|^2+|y|^2-|z|^2=-1
     \end{aligned} \right.
\end{equation}
However we will show that the solution for (\ref{eqnsys}) does not exist. To show this, first observe that $z\neq 0$ and so,
$$\bar z x i \bar x z + \bar z y i \bar y z - |z|^4 i=0.$$
Put $$t=\frac{\bar z x}{|z|^2}, \ s=\frac{\bar z y}{|z|^2}.$$
Then the equation system (\ref{eqnsys}) is reformulated to 
\begin{equation}\label{eqnsys2}
\left\{\begin{aligned}
& t i \bar t +s i \bar s - i=0 \\
& |z|^2(|t|^2+|s|^2-1)=-1
     \end{aligned} \right.
\end{equation}
Let $t=t_1+t_2i+t_3j+t_4k$ and $s=s_1+s_2i+s_3j+s_4k$. From the first equation in (\ref{eqnsys2}), we have 
$$(t_1^2+t_2^2-t_3^2-t_4^2)+(s_1^2+s_2^2-s_3^2-s_4^2)=1.$$
The second equation in (\ref{eqnsys2}) implies $|t|^2+|s|^2-1<0$. Then
\begin{eqnarray*}
0 &>&|t|^2+|s|^2-1 \\
&=& (t_1^2+t_2^2+t_3^2+t_4^2)+(s_1^2+s_2^2+s_3^2+s_4^2)-1 \\
&=& (t_1^2+t_2^2+s_1^2+s_2^2)+(t_3^2+t_4^2+s_3^2+s_4^2)-1 \\
&=& (t_3^2+t_4^2+s_3^2+s_4^2)+1+(t_3^2+t_4^2+s_3^2+s_4^2)-1 \\
&=& 2(t_3^2+t_4^2+s_3^2+s_4^2)
\end{eqnarray*}
This is impossible. Thus there does not exist the solution for (\ref{eqnsys}). Finally we can conclude that for any $g\in \mathrm{Sp}(2,1)$, the set of traces of $g \mathrm{SU}(2,1) g^{-1}$ is not contained in $\mathbb R$, which completes the proof.
\end{proof}

Now the cases of $\mathrm{SO}(2,1)$ and $1 \oplus \suone$ remain. In either case, every element has real trace. Hence it is possible that $L_{nc}$ is conjugate to either $\mathrm{SO}(2,1)$ or $1 \oplus \suone$. 

\begin{theorem}\label{sp21thm}
Let $\Gamma$ be a nonelementary discrete subgroup of $\mathrm{Sp}(2,1)$ with real trace field. Then $\Gamma$ stabilizes a totally geodesic submanifold in $\mathbf{H}_{\mathbb H}^2$ which is isometric to either $\mathbf{H}_{\mathbb R}^2$ or $\mathbf{H}_{\mathbb C}^1$.
\end{theorem}
\begin{proof}
As seen above, $L_{nc}$ is conjugate to either $\mathrm{SO}(2,1)$ or $1 \oplus \suone$. Then $L_{nc}$ stabilizes a totally geodesic submanifold in $\mathbf{H}_{\mathbb H}^2$ that is isometric to either $\mathbf{H}_{\mathbb R}^2$ or $\mathbf{H}_{\mathbb C}^1$. Hence by a similar proof as in Theorem \ref{realtrace}, the Zariski closure $\overline \Gamma$ of $\Gamma$ also stabilizes the totally geodesic submanifold. It completes the proof.
\end{proof}

\begin{remark}
In \cite{Kim13}, Kim assumed that $\Gamma$ contains a loxodromic element fixing $0$ and $\infty$. If the trace is invariant under conjugation, Kim's result is sufficient to apply his result to arbitrary nonelemenary quaternionic hyperbolic Kleinian groups. However, trace is not invariant under conjugation in $\Sp$ and thus, it is not easy to apply the method in \cite{Kim13} to arbitrary nonelementary quaternionic hyperbolic Kleinian groups even in the case of $\mathrm{Sp}(2,1)$.
In Theorem \ref{sp21thm}, we do not assume that $\Gamma$ contains a loxodromic element fixing $0$ and $\infty$. This is one advantage of our approach against the approach in \cite{Kim13} in characterizing nonelementary discrete subgroups of $\Sp$ with real trace fields.
Another advantage of our approach is that it makes it possible to deal this problem with the general case of $\Sp$. 
\end{remark}

\subsection{General case}

Recall that by the same argument as in the $\su$ case, the identity component $\overline \Gamma^\circ$ of the Zariski closure of $\Gamma$ is decomposed into $\overline \Gamma^\circ=LT$ as in Lemma \ref{qualemma1}. 
Then $L_{nc}$ is conjugate to either $I_{n-m}\oplus\mathrm{Sp}(m,1)$, $I_{n-m}\oplus\mathrm{SU}(m,1)$ or $I_{n-m}\oplus\mathrm{SO}(m,1)$ for some $m\geq 1$. Furthermore it is required that every element of $L_{nc}$ has real trace. Hence one can expect that the possible Lie group among them for $L$ is conjugate to either $I_{n-m}\oplus\mathrm{SO}(m,1)$ or $I_{n-1}\oplus\mathrm{SU}(1,1)$ like the case of $\su$ because the set of traces for the other Lie groups is not contained in $\mathbb R$. However, it seems not easy to check whether the set of traces is not contained in $\mathbb R$ or not for all groups conjugate to the other Lie groups. Forturnately, we find two criteria for $L_{nc}$.

Let $d_n$ denote a diagonal matrix of size $n+1$ with ordered diagonal entries, $1,\ldots,1, i=\sqrt{-1}$.
We obtain the first criterion for $L_{nc}$ as follows.

\begin{proposition}[Criterion I]\label{dnprop} Let $G$ be a subgroup of $\Sp$ containing $d_n$. Then for any $g\in \Sp$, the set of traces of $g G g^{-1}$ is not contained in $\mathbb R$. 
\end{proposition}

\begin{proof}
Let $a_{i,j}$ denote the $(i,j)$-entry of $g$. Then $g^{-1}$ is written as 
$$g^{-1} = \left[ \begin{array}{cccc} \bar a_{1,1} & \cdots & \bar a_{n,1} & -\bar a_{n+1, 1} \\ \vdots & \ddots & \vdots &\vdots \\ \bar a_{1,n} & \cdots & \bar a_{n,n} & -\bar a_{n+1, n} \\ -\bar a_{1, n+1} & \cdots & - \bar a_{n, n+1} & \bar a_{n+1,n+1} \end{array} \right].$$
To prove the Theorem, it is sufficient to show that $tr(gd_ng^{-1})$ is not real for all $g\in \Sp$.
Assume that $tr(gd_ng^{-1}) \in \mathbb R$ for some $g\in \Sp$.
By a direct computation, $tr(gd_ng^{-1}) \in \mathbb R$ is equivalent to the following condition.
$$a_{1,n+1}i\bar a_{1,n+1}+\cdots + a_{n,n+1}i \bar a_{n,n+1} - a_{n+1,n+1}i\bar a_{n+1,n+1} \in \mathbb R.$$
In fact, since $\overline{qi\bar q}= -q i \bar q$ for any $q \in \mathbb H$, we have 
$$a_{1,n+1}i\bar a_{1,n+1}+\cdots + a_{n,n+1}i \bar a_{n,n+1} - a_{n+1,n+1}i\bar a_{n+1,n+1}=0.$$

From $g^*I_{n,1}g=I_{n,1}$, $$|a_{1,n+1}|^2+\cdots |a_{n,n+1}|^2 - |a_{n+1,n+1}|^2=-1.$$
Hence $(a_{1,n+1},\ldots, a_{n+1,n+1}) \in \mathbb H^{n+1}$ is a solution of the following equation system for quaternion numbers:
\begin{equation}\label{eqnsys_spn}
\left\{\begin{aligned}
& x_1i \bar x_1 + \cdots +x_n i \bar x_n - x_{n+1}i \bar x_{n+1}=0 \\
& |x_1|^2+\cdots+|x_n|^2-|x_{n+1}|^2=-1
     \end{aligned} \right.
\end{equation}
Similarly to the $\mathrm{Sp}(2,1)$ case, we will prove that the solution for (\ref{eqnsys_spn}) does not exist in $\mathbb H^{n+1}$.
From the second equation in (\ref{eqnsys_spn}), it follows that $|x_{n+1}|\geq 1$ and so $x_{n+1} \neq 0$.
Putting $t_m=\frac{\bar x_{n+1} x_m}{|x_{n+1}|^2}$ for each $m=1,\ldots,n$, the equation system (\ref{eqnsys_spn}) is reformulated as follows.
\begin{equation}\label{eqnsys_spn2}
\left\{\begin{aligned}
& t_1i \bar t_1 + \cdots +t_n i \bar t_n - i =0 \\
& |x_{n+1}|^2(|t_1|^2+\cdots+|t_n|^2-1)=-1
     \end{aligned} \right.
\end{equation}

Let $t_m=t_{m,1} + t_{m,2}i +t_{m,3}j+t_{m,4}k$ for $m=1,\ldots, n$. Then the first equation in (\ref{eqnsys_spn2}) gives rise to 
$$\sum_{m=1}^n\left( t_{m,1}^2+t_{m,2}^2 \right) - \sum_{m=1}^n \left( t_{m,3}^2+t_{m,4}^2 \right) =1.$$
At the same time, the second equation in (\ref{eqnsys_spn2}) gives rise to 
$$\sum_{m=1}^n\left( t_{m,1}^2+t_{m,2}^2 \right) + \sum_{m=1}^n \left( t_{m,3}^2+t_{m,4}^2 \right) -1 <0.$$
Thus we finally get $$2\sum_{m=1}^n \left( t_{m,3}^2+t_{m,4}^2 \right) <0.$$
This is impossible. Therefore, there does not exist the solution satisfying (\ref{eqnsys_spn}). This implies that $tr(gd_ng^{-1})$ can not be a real number for any $g\in \Sp$, which completes the proof.
\end{proof}

\begin{corollary}\label{spm1}
Let $m \geq 1$.
Then for any subgroup of $\Sp$ conjugate to $I_{n-m}\oplus\mathrm{Sp}(m,1)$, the set of traces  is not contained in $\mathbb R$.
\end{corollary}

\begin{proof}
Clearly, $d_n \in I_{n-m}\oplus\mathrm{Sp}(m,1)$ for any $1\leq m\leq n$. By Proposition \ref{dnprop}, the Corollary immediately follows. 
\end{proof}

Criterion I does not give any information for $I_{n-m}\oplus \mathrm{SU}(m,1)$ because of $d_n \notin I_{n-m}\oplus \mathrm{SU}(m,1)$ for any $m\geq 1$. We need another criterion. 
Let
$$c_1=\left[ \begin{array}{rrr} 1&0&0 \\ 0&i&0 \\ 0&0&-i \end{array} \right], \ c_2=\left[ \begin{array}{rrr} i&0&0 \\ 0&1&0 \\ 0&0&-i \end{array} \right], \ c_3=\left[ \begin{array}{rrr} i&0&0 \\ 0&-i&0 \\ 0&0&1 \end{array} \right].$$
Note that $c_1$, $c_2$ and $c_3$ all are elements of $\mathrm{SU}(2,1)$. The second criterion is as follows.

\begin{proposition}[Criterion II]\label{c123}
Let $G$ be a subgroup of $\Sp$ containing $I_{n-2}\oplus c_1$, $I_{n-2}\oplus c_2$ and $I_{n-2}\oplus c_3$. Then for any $g\in \Sp$, the set of traces of $gGg^{-1}$ is not contained in $\mathbb R$.
\end{proposition}

\begin{proof}
Suppose that there exists a $g\in \Sp$ such that the set of traces of $gGg^{-1}$ is contained in $\mathbb R$. Since  $I_{n-2}\oplus c_1$, $I_{n-2}\oplus c_2$ and $I_{n-2}\oplus c_3$ are elements of $G$, the traces of their conjugates by $g$ are real. By a direct computation, $tr(g (I_{n-2}\oplus c_1)g^{-1}) \in \mathbb R$ is equivalent to
$$\lambda_{n} + \lambda_{n+1} \in \mathbb R$$
where $\lambda_m = a_{1,m}i \bar a_{1,m}+\cdots +a_{n,m}i\bar a_{n,m} - a_{n+1,m}i\bar a_{n+1,m}$ for each $m$. Similarly, the conditions $tr(g (I_{n-2}\oplus c_2)g^{-1}) \in \mathbb R$ and $tr(g (I_{n-2}\oplus c_3)g^{-1}) \in \mathbb R$ are equivalent to $\lambda_{n-1}+\lambda_{n+1} \in \mathbb R$ and $\lambda_{n-1}+\lambda_n \in \mathbb R$ respectively. Thus, $\lambda_{n-1}$, $\lambda_n$ and $\lambda_{n+1}$ are real numbers. Moreover, $\bar \lambda_m = -\lambda_m$ for each $m=n-1,n,n+1$. Hence we have 
$$\lambda_{n-1}=\lambda_n=\lambda_{n+1}=0.$$

Notice that $\lambda_{n+1}=0$ implies that $(a_{1,n+1},\ldots, a_{n+1,n+1}) \in \mathbb H^{n+1}$ is a solution of the equation system (\ref{eqnsys_spn}). However, by the proof of Proposition \ref{dnprop}, there does not exist the solution for (\ref{eqnsys_spn}). Therefore we can conclude that there does not exist $g \in \Sp$ satisfying that the set of traces of $gGg^{-1}$ is contained in $\mathbb R$. 
\end{proof}

\begin{corollary}\label{sum1}
Let $m \geq 2$.
Then for any subgroup of $\Sp$ conjugate to $I_{n-m}\oplus\mathrm{SU}(m,1)$, the set of traces is not contained in $\mathbb R$.
\end{corollary}

\begin{proof}
Note that $I_{n-2}\oplus c_1$, $I_{n-2}\oplus c_2$ and $I_{n-2}\oplus c_3$ are elements of $I_{n-2}\oplus \mathrm{SU}(2,1)$. For $m\geq 2$, since  $I_{n-2}\oplus \mathrm{SU}(2,1) \subset I_{n-m}\oplus \mathrm{SU}(m,1)$, it follows that $I_{n-2}\oplus c_1$, $I_{n-2}\oplus c_2$ and $I_{n-2}\oplus c_3$ are also elements of $I_{n-m}\oplus \mathrm{SU}(m,1)$ for $m\geq 2$.
Hence the Corollary follows from Proposition \ref{c123}.
\end{proof}

Due to Criterion I and II, we can conclude that the possible simple Lie group of rank $1$ for $L_{nc}$ is conjugate to either $I_{n-m}\oplus\mathrm{SO}(m,1)$ for $m\geq 2$ or $I_
{n-1}\oplus\mathrm{SU}(1,1)$. 

\begin{theorem}\label{L2}
Let $\Gamma$ be a nonelementary discrete subgroup of $\Sp$ with real trace field. Then $\Gamma$ is conjugate to a subgroup of $\mathrm{Sp}(n-m)\oplus\mathrm{O}(m,1)$ for $m\geq 2$ or $\mathrm{Sp}(n-1)\oplus\suone$.
\end{theorem}
\begin{proof}
As mentioned before, $L_{nc}$ is conjugate to either $I_{n-m}\oplus\mathrm{SO}(m,1)$ for $m\geq 2$ or $I_{n-1}\oplus\suone$. 
Then $L$ preserves a totally geodesic submanifold $Y$ of $\mathbf{H}_{\mathbb H}^n$ which is isometric to either $\mathbf{H}_{\mathbb R}^m$ for $m\geq 2$ or $\mathbf{H}_{\mathbb C}^1$. Since $L$ and $T$ centralize each other, $T$ also preserves $Y$ and thus $\overline \Gamma^\circ$ preserves $Y$. Finally $\overline \Gamma$ preserves $Y$ since $\overline \Gamma^\circ$ is a normal subgroup of $\overline \Gamma$. Hence we conclude that $\Gamma$ preserves a totally geodesic submanifold in $\mathbf{H}_{\mathbb H}^n$ which is isometric to $\mathbf{H}_{\mathbb R}^m$ for $m\geq 2$ or $\mathbf{H}_{\mathbb C}^1$. The stabilizer subgroup of $\mathbf H_{\mathbb R}^m$ in $\Sp$ is conjugate to $\mathrm{Sp}(n-m)\oplus\mathrm{O}(m,1)$ and the stabilizer subgroup of $\mathbf{H}_{\mathbb C}^1$ in $\Sp$ is conjugate to $\mathrm{Sp}(n-1)\oplus\suone$. Hence we completes the proof.
\end{proof}

If $\Gamma$ is an irreducible discrete subgroup of $\Sp$ with real trace field, then $\Gamma$ should be conjugate to a Zariski dense subgroup of $\mathrm{O}(n,1)$ by Theorem \ref{L2}. Hence Theorem \ref{thm:1.6} follows for the $\Sp$ case.

\begin{proof}[Proof of Theorem \ref{thm:1.2}]
It follows from Theorem \ref{L2} that $\Gamma$ preserves a totally geodesic submanifold in $\mathbf H_{\mathbb H}^n$ which is isometric to one of $\mathbf{H}_{\mathbb R}^m$ for $m\geq 2$ and $\mathbf{H}_{\mathbb C}^1$. They are totally geodesic submanifolds of constant negative sectional curvature in $\mathbf H_{\mathbb H}^n$ which are not isometric to $\mathbf H_{\mathbb H}^1$. Thus Theorem \ref{thm:1.2} follows immediately.
\end{proof}

\begin{proof}[Proof of Theorem \ref{thm:1.3}]
It is sufficient to prove the theorem in the $\Sp$ case because of $\su \cap \mathrm{O}(n,1)=\mathrm{SO}(n,1)$.
Let $\rho :\Gamma \rightarrow \Sp$ be a discrete faithful representation such that the trace field of $\rho(\Gamma)$ is real. First assume that $\overline{\rho(\Gamma)} < \mathrm{Sp}(n-m) \oplus \mathrm{O}(m,1)$.
Then the representation $\rho$ is decomposed into $\rho_c \oplus \rho_{nc}$ where $\rho_c :\Gamma \rightarrow \mathrm{Sp}(n-m)$ and $\rho_{nc} : \Gamma \rightarrow \mathrm{O}(m,1)$. Since $\mathrm{Sp}(n-m)$ is compact, $\rho_{nc}$ should be a discrete faithful representation in order that $\rho$ is discrete and faithful. Thus we obtain a discrete faithful representation $\rho_{nc} : \Gamma \rightarrow \mathrm{O}(m,1)$. 

In the case that $\overline{\rho(\Gamma)} < \mathrm{Sp}(n-1) \oplus \mathrm{SU}(1,1)$, the projection of $\rho$ onto the $\mathrm{SU}(1,1)$ factor is a discrete faithful representation by a similar reason as the previous case. By composing the Lie group homomorphism from $\mathrm{SU}(1,1)$ to $\mathrm{SO}(2,1)$, we also obtain a discrete faithful representation $\Gamma \rightarrow \mathrm{SO}(2,1)$. Therefore (i) implies (ii). The converse is clear.
\end{proof}

\end{document}